\documentclass [12pt]{amsart}

\usepackage[utf8]{inputenc}

\usepackage{amssymb,amsmath,amsthm}
\newtheorem{theorem}{Theorem}
\newtheorem{lemma}{Lemma}
\newtheorem{proposition}{Proposition}

\newtheorem{corollary}{Corollary}
\newtheoremstyle{neosn}{0.5\topsep}{0.5\topsep}{\rm}{}{\bf}{.}{ }{\thmname{#1}\thmnumber{ #2}\thmnote{ {\mdseries#3}}}
\theoremstyle{neosn}
\newtheorem{definition}{Definition}
\newtheorem{example}{Example}

\newcommand{\ad}{\,\mathrm{ad}\,}

\renewenvironment{proof}{\noindent \textbf{Proof.}}{$\blacksquare$}
\newcommand{\GL}{\,\mathrm{GL}\,}
\newcommand{\SL}{\,\mathrm{SL}\,}
\newcommand{\PGL}{\,\mathrm{PGL}\,}
\newcommand{\PSL}{\,\mathrm{PSL}\,}

\newcommand{\UT}{\,\mathrm{UT}\,}
\newcommand{\Exp}{\,\mathrm{Exp}\,}
\newcommand{\Rad}{\,\mathrm{Rad}\,}

\begin{document}

\begin{center}

{\Large {\bf Regular bi-interpretability of Chevalley groups  \\

\medskip

over local rings}}

\bigskip
 {\large \bf Elena~Bunina}

\medskip

 {\large \bf{Bar-Ilan University}}

\end{center}
\bigskip

 \begin{center}

{\large{\bf Abstract}}

\end{center}

In this paper we prove that if  $G(R)=G_\pi (\Phi,R)$ $(E(R)=E_{\pi}(\Phi, R))$ is an (elementary)  Chevalley group of rank $> 1$, $R$ is a local ring  (with $\frac{1}{2}$ for the root systems ${\mathbf A}_2, {\mathbf B}_l, {\mathbf C}_l, {\mathbf F}_4, {\mathbf G}_2$ and with  $\frac{1}{3}$ for  ${\mathbf G}_{2})$, then the group $G(R)$ (or $(E(R)$) is regularly bi-interpretable with the ring~$R$.
As a consequence of this theorem, we show that the class of all Chevalley groups over local rings (with the listed restrictions) is elementary definable, i.\,e., if for an arbitrary group~$H$ we have $H\equiv G_\pi(\Phi, R)$, than there exists a ring $R'\equiv R$ such that $H\cong G_\pi(\Phi,R')$.

\bigskip

\section{Introduction, history and definitions}\leavevmode

\subsection{Elementary equivalence}\leavevmode

Two models  ${\mathcal U}$ and ${\mathcal U}'$ of the same first
order language~$\mathcal L$ (for example, two groups or two rings)
are called \emph{elementarily equivalent}, if every sentence~$\varphi$
of the language~$\mathcal L$ holds in~${\mathcal U}$ if and only if
it holds in~${\mathcal U}'$.  Any two finite models of the same
language are elementarily equivalent if and only if they are
isomorphic. Any two isomorphic models are elementarily equivalent, but
for infinite models the converse fact is not true. For example, the
field $\mathbb C$ of complex numbers and the field $\overline
{\mathbb Q}$ of algebraic numbers are elementarily equivalent, but not
isomorphic as they have different cardinalities (for more detailed examples
see, for example,~\cite{Marker}). 

Tarski and Maltsev pushed forward a problem of describing groups and rings (in some natural classes) that are elementarily equivalent. There were obtained several complete results for some classes of groups and rings: for example, two algebraically closed fields are elementarily equivalent if and only if they have the same characteristics (classical result); two Abelian groups are elementarily equivalent if and only if they have the same special ``characteristic numbers'' (like $\Exp A$, $\dim p^nA[p]$, etc.), which must be either the same finite numbers or infinity (Szmielew, \cite{Shmelew}); similar results with invariants  were obtained  for Boolean rings (Ershov--Tarski, \cite{Boolean}),  also there were obtained  several other classification results. 
An outstanding result was the answer for old questions that were raised by A.\,Tarski around 1945:
 all non-abelian free groups are elementary equivalent (O.\,Kharlampovich, A.\,Myasnikov \cite{Kharlamp-Myasikov}; Z.\,Sela~\cite{Sela}),.
Certainly, any new such ``pure classification result'' is a rather rare mathematical phenomenon.

\subsection{Maltsev-type theorems for linear groups}\leavevmode

Another way of studying elementary equivalence of algebraic models is to establish a connection between derivative models (linear groups over rings and fields, automorphism groups and endomorphism rigs of different structures, etc) and initial models (and ``parameters'') used for the construction.
First results in this field were obtained
by A.I.~Maltsev in 1961 in~\cite{Maltsev}. He proved that the groups
$\mathcal G_n(K_1)$ and $\mathcal G_m(K_2)$ (where $G=\GL,\SL,\PGL, \PSL$, $n,m\geqslant
3$, $K_1$,~$K_2$ are fields of characteristics~$0$) are elementarily
equivalent if and only if
 $m=n$ and the fields  $K_1$~and~$K_2$ are elementarily equivalent.

In 1961--1971 H.J.\,Keisler (\cite{Keisler}) and S.\,Shelah (\cite{Shelah}) proved the next Isomorphism theorem:

\begin{theorem}
Two models  $\mathcal U_1$ and $\mathcal U_2$ are elementary equivalent if and only if
there exists an ultrafilter  $\mathcal F$ such that their ultrapowers coincide:
$$
\prod_{\mathcal F} \mathcal U_1\cong \prod_{\mathcal F} \mathcal U_2.
$$
\end{theorem}

Using this Isomorphism theorem K.I.~Beidar and A.V.~Mikhalev in 1992  (see~\cite{BeidarMikhalev}) found the general approach to this type of theorems  and generalised Maltsev theorem for the case when  $K_1$ and $K_2$ are skewfields and prime associative rings.  This approach for the groups $\GL_n$ was generalised in~\cite{Bunina_Bragin}  in the following result:

\medskip

\begin{theorem}\label{finite number idempotents}
Let $R_1$ and $R_2$ be associative rings with~$1$
$(1/2)$ with finite number of central idempotents and $m,n\geqslant 4$ $(m,n\geqslant 3)$. Then $\GL_m(R_1)\equiv \GL_n(R_2)$ if and only if there exist  central idempotents $e\in R$  and $f \in S$ such  that $e M_m(R)\equiv f M_n(S)$ and $(1-e)M_m(R)\equiv (1-f) M_n(S)^{op}$.
\end{theorem} 

\medskip

Continuation of investigations in this field were the papers of E.I.~Bunina 1998--2010 (see~\cite{Bunina_intro1}, \cite{Bunina_intro3}, \cite{Bunina-local}, \cite{Bunina_intro2}), where the results of A.I.~Maltsev were extended for unitary linear groups over skewfields and associative rings with involution, and also for Chevalley groups over fields and local rings. The important for this paper is the following theorem (see~\cite{Bunina-local}):

\medskip

\begin{theorem}\label{th_Chevalley_Local}
Let
 $G=G_\pi (\Phi,R)$ and $G'=G_{\pi'}(\Phi',R')$
\emph{(}or $E_\pi (\Phi,R)$ and $E_{\pi'}(\Phi',R'))$ be two
\emph{(}elementary\emph{)} Chevalley groups over infinite local
rings $R$ and~$R'$ with two invertible (in the case of the root
system $G_2$ with three invertible), with indecomposable root
systems $\Phi,\Phi'$ of ranks $> 1$, with weight lattices $\Lambda$
and $\Lambda'$, respectively.
 Then the groups $G$ and $G'$ are elementarily equivalent if and only if the root systems  $\Phi$
 and $\Phi'$ are isomorphic, the rings $R$  and $R'$
are elementarily equivalent, the lattices $\Lambda$ and $\Lambda'$
coincide.
\end{theorem}

\medskip

This last theorem was proved partially using explicit first order formulas, partially with the usage of Keisler--Shelah Isomorphism theorem.  In 2019 this result was almost generalised for Chevalley groups over arbitrary commutative rings (using the algebraic results from~\cite{Chevalley_main} on classification of all automorphisms of Chevalley groups over commutative rings and Keisler--Shelah Isomorphism theorem):

\medskip 

\begin{theorem}[E.\,Bunina, \cite{BunChevComm}]\label{th_Chevalley_Comm}
Let
 $G=G_{\ad} (\Phi,R)$ and $G'=G_{\ad}(\Phi',R')$
\emph{(}or $E_{\ad} (\Phi,R)$ and $E_{\ad'}(\Phi',R'))$ be two adjoint
\emph{(}elementary\emph{)} Chevalley groups  with indecomposable root
systems $\Phi,\Phi'$ of ranks $> 1$ over infinite commutative
rings $R$ and~$R'$  (for the cases of the roots systems $A_2$, $B_l$, $C_l$, $F_4$ with $1/2$
and for the system $G_2$ with $1/2$ and $1/3$).
 
Then the groups $G$ and $G'$ are elementarily equivalent if and only if the root systems  $\Phi$
 and $\Phi'$ are isomorphic and  the rings $R$  and $R'$
are elementarily equivalent.
\end{theorem}

\subsection{Elementary definability}\leavevmode

Along with problems such as Maltsev's theorem, problems of so-called \emph{elementary definability} were also posed.
These problems are formulated as follows:

\smallskip

\emph{Suppose that we have a class $\mathcal G$ of groups (rings, etc.) and an arbitrary group (rings, etc.) $H$ which is elementarily equivalent to some $G\in \mathcal G$. Is it true that $H\in \mathcal G$?}

\smallskip

For example the class of all Abelian groups is clearly elementary definable. More generally {\bf any variety of groups is elementary definable}: if $\mathcal G$ is some variety of groups defined by a system of identity relations $\Lambda$ and a group $H$ is elementarily equivalent to some $G\in \mathcal G$, then $H$ satisfies the same system $\Lambda$, therefore, $H\in \mathcal G$. 

The next interesting example is that {\bf the class of all linear groups over fields is elementary definable}: 

\begin{theorem}[A.I.\,Maltsev,~\cite{Mal}]
If $H$ is a group that is elementarily equivalent to a linear group, then $H$ is linear.
\end{theorem}

For the sake of completeness we include a proof.

\begin{proof}
Let $G \subset \GL_n(K)$ be a linear group which is elementary equivalent to~$H$. Enumerate
the elements of~$H$ as $H=\{ h_\alpha \mid \alpha\in \varkappa\}$ and enumerate all relations that hold between $h_\alpha$ by $\{ r_\beta (h_\alpha) \mid \beta\in \mu\}$. 
 Let $\mathcal L$ be the first-order language of rings together with constants  $c^\alpha_{ij}$ for $\alpha \in \varkappa$ and $1\leqslant i,j\leqslant n$.
 Consider the theory $T$ consisting of the following statements:

(1) The axioms of fields.

(2) The statements $\det(c^\alpha_{i,j})\ne 0$ for all $\alpha \in \varkappa$.

(3) The statements $(c^\alpha_{i,j})\ne  (c^\beta_{i,j})$ for all distinct $\alpha,\beta\in \varkappa$.

(4) The statements $r_\beta((c^\alpha_{i,j}))=1$ for all $\beta\in \mu$.

If $S$ is a finite subset of~$T$, then there is a finite set $I_0\subset \varkappa$ such that $S$ involves only
$c^\alpha_{i,j}$ for $\alpha \in I_0$. Since the elements $h_\alpha$, $\alpha \in I_0$,  satisfy all relations $r_\beta$ that involve only them, we get that there are elements 
$(g^\alpha_{i,j})\in G\subset \GL_n(K)$
that satisfy~$S$. In particular, $S$ is consistent. By the Compactness Theorem, there is a model $\mathbb K$
of~$T$. This $\mathbb K$ must be a field, and the map $h_\alpha \mapsto (c^\alpha_{i,j})$ is an embedding
$H\to \GL_n(\mathbb K)$.
\end{proof}

It is easy to see that the same proof holds for linear groups over commutative rings (is we just the axioms of commutative rings instead of the axioms of fields in (1)) or over any elementarily axiomatized class of commutative rings. Therefore  {\bf the class of all linear groups over commutative rings is elementary definable}.

On the other hand for general linear groups even over fields there is no elementary definability: there exists a group $H\equiv \GL_n(R)$ which is not $\GL$ itself.
To demonstrate it we will show an easy example.

\begin{example}
Consider the group  $\GL_3(\mathbb R)$ and construct elementarily equivalent group which is not of the type $\GL_3$.

To construct it we will use the following well-known statement (see, for example, \cite{KeisChan}):

\emph{Filtered products, filtered powers, direct products and direct powers preserve elementary equivalence. }

Consider now  the field $\mathbb R$ of real numbers and any arbitrary countable field $\mathbb K$ which is elementarily equivalent to~$\mathbb R$. Then consider the group $\SL_3(\mathbb R)\times \mathbb K^{*}$. We see that it is elementarily equivalent to the group $\SL_3(\mathbb R)\times \mathbb R^{*}$.

Let us now construct an isomorphism between $\GL_3(\mathbb{R})$ and $\SL_3(\mathbb R)\times \mathbb R^{*}$. Any matrix $X$ from $\GL_3(\mathbb{R})$ can be represented as the product of a matrix with determinant~$1$ and some scalar matrix with  $(\det(X))^{\frac{1}{3}}$ on its diagonal, since cubic root is bijective function in~$\mathbb{R}$.

Finally let us prove that $\SL_3(\mathbb R)\times \mathbb K^{*}$ is not a group of type $\GL_3(\mathbb{K}')$. Suppose that it is isomorphic to any $\GL_3(\mathbb K')$. Then from one side, considering centres of these groups we conclude that  $\mathbb{K}'$ has the same cardinality as $\mathbb{K}^{*}$ (i.\,e., it is countable), from another side, looking at quotient groups by their centres  we conclude that $\mathbb{K}'$ has the same power as $\mathbb{R}$ (i.\,e., has the power of continuum), contradiction.
\end{example}

\medskip

In 1984 B.I.\,Zilber (see~\cite{Zilber1984}) proved that for algebraic groups over algebraically closed fields  the problem of elementary definability has the positive answer:

\begin{theorem}\label{Zilber-definability}
The class of all simple algebraic groups over algebraically closed fields (or of all simple algebraic groups of the given type over algebraically closed fields) is elementary definable.
\end{theorem}

In fact this problem has positive answer for simple algebraic groups over a wider class of rings, we will mention later  on the corresponding results of A.\,Myasnikov and M.\,Sohrabi and prove a generalization of this theorem for local rings.

\subsection{First order rigidity}\leavevmode

A new round of development of this topic has appeared recently in the papers of A.\,Nies (\cite{Nies_main}, \cite{Nies19}, etc), N.\,Avni, A.\,Lubotzky, C.\,Meiri (\cite{MyasKharl3}, \cite{AvniMeqri2}, also see~\cite{MyasKharl30}), A.\,Myasnikov, O.\,Kharlampovich and M.\,Sohrabi (\cite{Myas-Kharl-Sohr}, \cite{Myasnikov-Sohrabi}, \cite{Myasnikov-Sohrabi2}, etc.),  
D.\,Segal and K.\,Tent (\cite{Segal-Tent}), B.\,Kunyavskii, E.\,Plotkin, N.\,Vavilov (\cite{KuPlVav}) and others.

Since according to L\"owenheim--Skolem theorem
for any infinite structure $\mathcal A$ there exists a structure $\mathcal B$ such that $\mathcal A\equiv \mathcal B$ and they
have different cardinalities, so classification up to elementary equivalence is always strictly wider than classification up to isomorphism. But taking  finitely generated  structures (groups or   rings or anything else) there exists a possibility that for some of them elementary equivalence could imply isomorphism.

For example it is true for the ring $\mathbb Z$ of integers: if any finitely generated ring $A$ is elementarily equivalent to~$\mathbb Z$, then it is isomorphic to~$\mathbb Z$, but it is not true for infinitely generated rigs.

A question dominating research in this area has been if and when elementary
equivalence between finitely generated groups (rings) implies isomorphism. Recently,
Avni et. al.~\cite{MyasKharl3} presented the term \emph{first-order rigidity}: 

\medskip

\begin{definition}
A finitely generated
group (ring or other structure) $A$ is \emph{first-order rigid} if any other finitely generated group (ring or other structure)
elementarily equivalent to~$A$ is isomorphic to~$A$.
\end{definition}

\medskip

There are many examples of first-order rigid groups and algebras. For instance,
Avni, Lubotzky and Meiri~\cite{MyasKharl3} showed that non-uniform higher dimensional
lattices are first-order rigid, as well as finitely generated pro-finite groups~\cite{MyasKharl30}. 
Lasserre showed that under some natural conditions polycyclic groups~\cite{MyasKharl28}
are also first-order rigid, etc.  

\subsection{QFA-property}\leavevmode

In many groups holds the more strong property
when a single first-order axiom is sufficient to distinguish a group among
all the finitely generated groups. The following concept was introduced in \cite{Nies19} for
groups only, and then in \cite{Nies_main} for arbitrary structures:
 
\medskip

\begin{definition}\label{QFA}
 Fix a finite signature. An infinite finitely generated structure $\mathcal U$ is \emph{quasi finitely
axiomatizable} (\emph{QFA}) if there is a first order sentence $\varphi$ such that

(1) $\mathcal U\vDash \varphi$;

(2)  if $H$ is a finitely generated structure in the same signature such that $H \vDash \varphi$, then
$H\cong \mathcal U$.
\end{definition}

\medskip

These are examples of QFA groups (see~\cite{Nies_main}):

(1) nilpotent groups $\UT_3(\mathbb Z)$;

(2) metabeilan groups ${\mathbb Z}[1/m] \rtimes {\mathbb Z} =\langle a,d \mid d^{-1} ad=a^m\text{ for any }m\geqslant 3\rangle$ and $\mathbb Z/p \mathbb Z\wr \mathbb Z$ for any prime~$p$;  

(3) permutation groups: the subgroup of the group of permutations of~$\mathbb Z$
generated by the successor function and the transposition $(0, 1)$.

Speaking about rings the best (and the most important) example of QFA rings is the ring of integer $(\mathbb Z, +, \cdot)$ (\cite{Sabbagh}).

\subsection{Interpretability and bi-interpretability}\leavevmode

The real breakthrough in this area was the application of {\bf bi-interpretability}  towards finding new classes of  QFA rings and groups.

There is well-known notion of interpretability (with or without parameters) one structure in another one.
Roughly speaking, $\mathcal B$ is interpretable in~$\mathcal A$ if
the elements of~$\mathcal B$ can be represented by tuples in a definable relation~$\mathcal D$ on~$\mathcal A$,
in such a way that equality of~$\mathcal B$ becomes an $\mathcal A$-definable equivalence relation
$\mathcal E$ on~$\mathcal D$, and the other atomic relations on~$\mathcal B$ are also definable. 

Remark that all first-order formulas in these notions are allowed to contain parameters.

\medskip

\begin{example}
A simple
example is the difference group construction: for instance, $(\mathbb Z, +)$ can be
interpreted in $(\mathbb N, +)$, where the relation $\mathcal D$ is $\mathbb N\times \mathbb N$, addition is 
component-wise and $\mathcal E$ is the relation given by $(n,m)\mathcal E (n',m')\Longleftrightarrow n'+m=n+m'$.
Further examples include the quotient fields, which can be interpreted in the
given integral domain, and the group $\GL_n(R)$ for fixed $n\geqslant 1$, which can be
first-order interpreted in the ring~$R$. 
\end{example}

\medskip

There is a very  important notion of bi-interpretability.

We will consider the isomorphic
copy of~$\mathcal B$ that is defined in~$\mathcal A$ by the relevant collection of formulas. For
instance, taking in account the example of an interpretation of $(\mathbb Z, +)$ in $(\mathbb N, +)$ above; the actual copy of $(\mathbb Z,+)$ defined in $(\mathbb N,+)$ is the structure whose
domain consists of pairs of natural numbers, whose addition is component-wise
and where equality is the equivalence relation on pairs given there.

Suppose structures $\mathcal A, \mathcal B$ in finite signatures are given, as well as
interpretations of $\mathcal A$ in~$\mathcal B$, and vice versa. Then an isomorphic copy~$\widetilde{\mathcal A}$ of~$\mathcal A$ can be defined in~$\mathcal A$, by ``decoding''~$\mathcal A$ from the copy of~$\mathcal B$ defined in~$\mathcal A$.

Similarly, an isomorphic copy $\widetilde{\mathcal B}$ of~$\mathcal B$ can be defined in~$\mathcal B$. An isomorphism $\Phi: {\mathcal A}\cong \widetilde{\mathcal A}$ can be viewed as a relation on~$\mathcal A$, and similarly for an isomorphism $\mathcal B \cong \widetilde{\mathcal B}$. Let us give a variant of a notion of bi-interpretability:

\medskip

\begin{definition}[(see~\cite{Nies7}, Ch.\,5)]
We say that $\mathcal A$ and $\mathcal B$ are \emph{bi-interpretable} (with parameters) if there
are such isomorphisms that are first-order definable. 
\end{definition}

\medskip

Along with the notion of bi-interpretability with parameters it is crucial to introduce a notion of {\bf regular bi-interpretability}.

\medskip

\begin{definition}
We say that a structure $\mathcal A$ is interpreted in a given structure~$\mathcal B$ \emph{uniformly}
with respect to a subset $D\subseteq \mathcal B^k$ if there is one interpretation of~$\mathcal A$ in~$\mathcal B$ with any tuple of parameters $\overline p\in D$.

 If $\mathcal A$ is interpreted in~$\mathcal B$ uniformly with respect
to a $\emptyset$-definable subset $D\subseteq \mathcal B^k$ then we say that $\mathcal A$ is \emph{regularly interpretable} in~$\mathcal B$ and write in this case $\mathcal A \cong \Gamma(\mathcal B, \varphi)$, provided $D$ is defined by~$\varphi$ in~$\mathcal B$. 

Note that the interpretability without parameters (\emph{absolute interpretability}) is a particular case of the regular interpretability where the set $D$ is empty.
\end{definition}

\medskip

An  important application of regular interpretability is the following:

\begin{proposition}
If $\mathcal A_1 =\Gamma (\mathcal B_1,\varphi)$ and $\mathcal A_2=\Gamma (\mathcal B_2,\varphi)$ are two regular interpretations with the same interpretation and the same formula~$\varphi$, then $\mathcal B_1\equiv \mathcal B_2$ implies $\mathcal A_1\equiv \mathcal A_2$.
\end{proposition}

\medskip

In fact all Maltsev-type theorems mentioned above were proved either by Keisler--Shelah Isomorphism theorem or by regular interpretability of a basic field/ring in a corresponding linear groups, including the initial Maltsev theorem (Theorem~1).

\medskip

\begin{definition}
Two algebraic structures $\mathcal A$ and $\mathcal B$ are called \emph{regularly bi-interpretable}
in each other if the following conditions hold:

(1) $\mathcal A$ and $\mathcal B$ are regularly interpretable in each other, so $\mathcal A \cong \Gamma(\mathcal B, \varphi)$ and $\mathcal B \cong  \Delta(\mathcal A, \psi)$ for some interpretations (also called \emph{codes})  $\Gamma$ and $\Delta$ and the corresponding formulas $\varphi, \psi$
(without parameters). By transitivity $\mathcal A$, as well as~$\mathcal B$, is regularly interpretable in itself, so we can write
$$
 \mathcal A \cong (\Gamma \circ \Delta)(\mathcal A, \varphi^*)\text{ and }
\mathcal B \cong  (\Delta \circ \Gamma )(\mathcal B, \psi^*),
$$
 where $\circ$ denotes composition of interpretations and $\varphi^*, \psi^*$ are the corresponding formulas.

(2) There is a formula $\theta (\overline y, x, \overline z)$ in the language of~$\mathcal A$
 such that for every tuple~$p^*$ satisfying $\varphi^*(\overline z)$ in~$\mathcal A$ the formula $\theta (\overline y,x, p^*)$   defines in~$\mathcal A$ the isomorphism $\overline \mu_{\Gamma \circ \Delta}:  (\Gamma \circ \Delta)(\mathcal A, p^*) \to \mathcal A$ and there is a formula $\sigma (\overline u, x, \overline )$ in the
language of~$\mathcal B$ such that for every tuple $q^*$ satisfying $\psi^*(\overline v)$  in~$\mathcal B$ the formula
$\sigma (\overline v, x, q^*)$ defines in~$\mathcal B$ the isomorphism $\overline \mu_{\Delta \circ \Gamma}:  (\Delta \circ \Gamma)(\mathcal B, q^*) \to \mathcal B$.
\end{definition}

\medskip

If linear groups (or another derivative structures) of any concrete types over some classes of fields/rings are regularly bi-interpretable with the corresponding rings, then this class of groups (structures) is elementary definable (see~Proposition~\ref{prop-regular->def}).

\subsection{Bi-interpretability with the ring $\mathbb Z$ and QFA property}\leavevmode

The special interesting case is when one can prove that a given structure is bi-interpretable with the ring $\mathbb Z$ of integers.

Khelif~\cite{Nies12} (see also~\cite{Myas-Kharl-Sohr}) realised that one can use bi-interpretability of a finitely generated structure~$\mathcal A$ with~$(\mathbb Z, + \times)$ as a general method
to prove that $\mathcal A$ is QFA. 
Somewhat later, Scanlon independently
used this method to show that each finitely generated  field is QFA.

\begin{theorem}[Nies, \cite{Nies12}]
 Suppose the structure $\mathcal A$ in a finite signature is bi-interpretable
with the ring~$\mathbb Z$. Then $\mathcal A$ is QFA.
\end{theorem}

By a \emph{number field} we mean a finite extension of~$\mathbb Q$. By the \emph{ring of integers}~$\mathcal O$ of a number field~$F$ we mean the subring of~$F$ consisting of all roots of monic polynomials with
integer coefficients.

By  a known result any ring of integers $\mathcal O$ of a number field~$F$ is bi-interpretable with~$\mathbb Z$ (and therefore is QFA).

In the paper~\cite{Myasnikov-Sohrabi} A.\,Myasnikov and M.\,Sohrabi considered linear groups ($\GL_n(\mathcal O)$, $\SL_n(\mathcal O)$ and triangular groups $T_n(\mathcal O)$) over rings of integers of a field~$F$.
They proved that the groups $\SL_n(\mathcal O)$, $n\geqslant 3$ are bi-interpretable (with parameters) with the ring~$\mathcal O$ and therefore with~$\mathbb Z$; the groups $\GL_n(\mathcal O)$ and $T_n(\mathcal O)$, $n\geqslant 3$, are bi-interpretable with~$\mathcal O$ (and with~$\mathbb Z$) only if $\mathcal O^*$ is finite.
Consequently in all good cases these linear groups are QFA. 

Moreover they proved that for all these ``good cases'' the problem of general elementary definability has positive solution:

\begin{theorem}[Myasnikov, Sohrabi, \cite{Myasnikov-Sohrabi}]
 If $n\geqslant 3$, $\mathcal O$ is the ring of integers of some number field~$F$, $H$ is an arbitrary groups, $H\equiv \SL_n(\mathcal O)$ or $\mathcal O^*$ is finite and $H\equiv \GL_n(\mathcal O)$ or $H\equiv T_n(\mathcal O)$. Then $H\cong \SL_n(R)$ (or $\GL_n(R)$, $T_n(R)$ respectively) for some ring~$R$ such that $R\equiv \mathcal O$.
\end{theorem}

Really they proved in their work that these groups are {\bf regularly} bi-interpretable with the corresponding rings.

The similar result was obtained by D.\,Segal and K.\,Tent~\cite{Segal-Tent}. They proved the following theorem about bi-interpretability (with parameters) of Chevalley groups over integral domains:

\begin{theorem}
 Let $G(\cdot)$ be a simple Chevalley-Demazure group scheme of rank at
least two, and let $R$ be an integral domain. Then $R$ and $G(R)$ are bi-interpretable
provided either

\emph{(1)} $G$ is adjoint, or

\emph{(2)} $G(R)$ has finite elementary width,\\
assuming in case $G$ is of type ${\mathbf E}_6$, ${\mathbf E}_7$, ${\mathbf E}_8$, or ${\mathbf F}_4$ that $R$ has at least two units.
\end{theorem}

From this theorem the authors by the similar way as Myasnikov--Sohraby proved the corollary about rigidity/QFA-property of Chevalley groups:

\begin{corollary}
Assume that $G$ and $R$ satisfy the hypotheses of the previous theorem. If $R$ is first order rigid (resp., QFA), in

\emph{(1)} the class o finitely generated rings,

\emph{(2)} the class of profinite rings, 

\emph{(3)} the class of locally compact  topological rings,\\
 then $G(R)$ has the analogous property in \emph{(1)} the class of finitely generated groups, \emph{(2)}
the class of profinite groups, \emph{(3)} the class of locally compact topological
groups.
\end{corollary}

All these recent papers illustrate that together with bi-interpretability with parameters it is unprecedented important to state regular bi-interpretability of linear groups and corresponding rings.

Therefore this paper is devoted to {\bf regular bi-interpretabilty of Chevalley groups over local rings}. As a corollary we will obtain elementary definability of these groups in the class of all groups.

We prove the following theorem and its corollary:

\begin{theorem}\label{th-reg-bi-interpret}
If  $G(R)=G_\pi (\Phi,R)$ $(E(R)=E_{\pi}(\Phi, R))$ is an (elementary)  Chevalley group of rank $> 1$, $R$ is a local ring  (with $\frac{1}{2}$ for the root systems ${\mathbf B}_l, {\mathbf C}_l, {\mathbf F}_4, {\mathbf G}_2$ and with  $\frac{1}{3}$ for  ${\mathbf G}_{2})$, then the group $G(R)$ (or $(E(R)$) is regularly bi-interpretable with~$R$.
\end{theorem}

\begin{corollary}[elementary definability of Chevalley groups]
The class of Chevalley groups over local rings is elementarily definable, i.\,e., if $G(R)=G_\pi (\Phi,R)$ is a Chevalley group of rank $> 1$, $R$ over a  local ring $R$  (with $\frac{1}{2}$ for the root systems ${\mathbf A}_2, {\mathbf B}_l, {\mathbf C}_l, {\mathbf F}_4, {\mathbf G}_2$ and with  $\frac{1}{3}$ for  ${\mathbf G}_{2})$  and for an arbitrary group $H$ we have $H\equiv G(R)$, then $H\equiv G_\pi(\Phi,R')$ for some local ring~$R'$, which is elementarily equivalent to~$R$.
\end{corollary}

From the next section we will concentrate specially on Chevalley groups over local rings. 
We start with necessary definitions and references.

\section{Chevalley groups over local rings and their properties}\leavevmode

Basic notions about root systems, semisimple Lie algebras, Chevalley
groups, that will be used in this paper, can be found in the papers
of the author \cite{Bunina_intro3}, \cite{Chevalley_main}.
Detailed information about root systems can be found in the books
\cite{Hamfris},~\cite{Burbaki}. More detailed information about
semisimple Lie algebras can be found in the book~\cite{Hamfris}.
More detailed information about elementary Chevalley groups is
contained in the book~\cite{Steinberg}, and about the Chevalley
groups  (also over rings) in~\cite{Vavilov}, \cite{VavPlotk1}, \cite{Chevalley} (see also
later references in these papers).

We fix some arbitrary (indecomposable) root system $\Phi$ of the
rank~$l\geqslant 2$, we suppose that in this system there are
$n$~positive and  $n$~negative roots.

Additionally we fix some infinite commutative local ring~$R$ with~$1$
(for the root systems ${\mathbf A}_2$, ${\mathbf B}_l$, ${\mathbf C}_l$, ${\mathbf F}_4$ and ${\mathbf G}_2$ with $1/2$ and for  ${\mathbf G}_2$ also with $1/3$).

We consider an arbitrary Chevalley group $G_\pi(\Phi,R)$,
constructed by the root system $\Phi$, a ring~$R$ and a
representation $\pi$ of the corresponding Lie algebra. It is known,
that Chevalley group is defined by the root system, the ring~$R$ and
the weight lattice of the representation~$\pi$. We will denote this
lattice by~$\Lambda$ or $\Lambda_\pi$. If we consider an elementary
Chevalley group, we denote it by $E_\pi(\Phi,R)$.

The subgroup of all diagonal (in a standard basis of weight vectors)
matrices of the Chevalley group $G_\pi(\Phi,R)$ is called the
\emph{standard maximal torus} of $G_\pi(\Phi,R)$ and is denoted by
$T_\pi(\Phi,R)$. This group is isomorphic to $Hom(\Lambda_\pi,
R^*)$.

 Let us denote by $h(\chi)$
an element of $T_\pi (\Phi,R)$, corresponding to the homomorphism
$\chi\in Hom (\Lambda(\pi),R^*)$.

In particular, $h_\alpha(u)=h(\chi_{\alpha,u})$ ($u\in R^*$, $\alpha
\in \Phi$), where
$$
\chi_{\alpha,u}: \lambda\mapsto u^{\langle \lambda,\alpha\rangle}\quad
(\lambda\in \Lambda_\pi),\text{ where }\langle \lambda,\alpha\rangle = \frac{2(\lambda,\alpha)}{(\alpha,\alpha)}.
$$

Connection between Chevalley groups and corresponding elementary
subgroups is an important problem in the theory of Chevalley
groups over rings. For elementary Chevalley groups there exists a
convenient system of generators $x_\alpha (\xi)$, $\alpha\in \Phi$,
$\xi\in R$, and all relations between these generators are well-known.
For general Chevalley groups it is not always true.

If $R$ is an algebraically closed field, then
$$
G_\pi (\Phi,R)=E_\pi (\Phi,R)
$$
for any representation~$\pi$. This equality is not true even for the
case of fields, which are not algebraically closed.

However if $G$ is a universal group and the ring $R$ is \emph{semilocal}
(i.e., contains only finite number of maximal ideals), then we have
the condition
$$
G_{sc}(\Phi,R)=E_{sc}(\Phi,R).
$$
\cite{M}, \cite{Abe1}, \cite{St3}, \cite{AS}.

Let us show the difference between Chevalley groups and their
elementary subgroups in the case when a ring $R$ is semilocal and
a corresponding Chevalley group is not universal.  In this case $G_\pi
(\Phi,R)=E_\pi(\Phi,R)T_\pi(\Phi,R)$] (see~\cite{Abe1}, \cite{AS},
\cite{M}), and the elements $h(\chi)$ are connected with elementary
generators by the formula
\begin{equation}\label{e4}
h(\chi)x_\beta (\xi)h(\chi)^{-1}=x_\beta (\chi(\beta)\xi).
\end{equation}

If $\Phi$ is an irreducible root system of a rank $l\geqslant
2$, then $E(\Phi,R)$ is always normal in $G(\Phi,R)$ (see~\cite{Tadei}, \cite{Hasrat-Vavilov}). In the case of
semilocal rings with~$1/2$ it is easy to show that
$$
[G(\Phi,R),G(\Phi,R)]=E(\Phi,R).
$$

However in the case $l=1$ the subgroup of elementary matrices
$E_2(R)=E_{sc}(A_1,R)$ is not necessarily normal in the special linear
group $SL_2(R)=G_{sc}(A_1,R)$ (see~\cite{Cn}, \cite{Sw},
\cite{Su1}).

We will use several facts proved in~\cite{Bunina-local}.

The next lemma from~\cite{Bunina-local} (see Lemma~1 there) is well-known (see also \cite{KuPlVav} for the exact boundary):

\begin{lemma}\label{commutant}
Let $G=G_\pi(\Phi,R)$ be a Chevalley group, $E=E_\pi(\Phi,R)$ be its
elementary subgroup, $R$ a semilocal ring (with $1/2$ for the root systems ${\mathbf A}_2$, ${\mathbf B}_l$, ${\mathbf C}_l$, ${\mathbf F}_4$, ${\mathbf G}_2$  and with $1/3$  for  ${\mathbf G}_2$). Then $E=[G,G]$ and there
exists such a number~$N$, depending of~$\Phi$, but not of~$R$
\emph{(}and not of a representation~$\pi$\emph{)}, that every
element of the group~$E$ is a product of not more than
$N$~commutators of the group~$G$.
\end{lemma}

This lemma has a very important corollary:

\begin{corollary}\label{G-E}
Under assumptions of the previous lemma the elementary subgroup $E=E_\pi(\Phi,R)$ is interpretable in a Chevalley group $G=G_\pi(\Phi,R)$ without parameters with the same interpretation code for all Chevalley groups of the same types. 

In particular, if two Chevalley groups of the same type are elementarily equivalent, then their elementary subgroups are elementarily equivalent too.
\end{corollary}

An elementary adjoint Chevalley group $E_{\ad} (\Phi, R)$ is always the quotient group of $E_{\pi}(\Phi, R)$ by its center. Therefore the adjoint Chevalley group $E_{\ad} (\Phi, R)$ is absolutely interpretable in the initial Chevalley group $G=G_\pi(\Phi,R)$  for our case, so if  two Chevalley groups of the same type are elementarily equivalent, then the corresponding adjoint elementary Chevalley groups are elementarily equivalent too.

The next important introductory step is to show that the subgroup $E_J=E_{\ad}(\Phi,R,J)$, where $J$ is the radical of~$R$, is absolutely definable in $E_{\ad}(\Phi,R)$ (and so in the initial Chevalley group $G_\pi(\Phi,R)$). This group is generated by all $x_\alpha(t)$, $\alpha \in \Phi$, $t\in J$, 
and is the greatest (unique maximal) proper normal
subgroup of $E= E_{\ad}(\Phi, R)$ (see~\cite{A6}).

The following proposition was proved in~\cite{Bunina-local}:

\begin{proposition}[see Proposition~3 from~\cite{Bunina-local}]
The subgroup $E_J=E_{\ad}(\Phi,R,J)$ is absolutely definable in
$E=E_{\ad}(\Phi,R)$, if $R$ is as in Lemma~\ref{commutant}.
\end{proposition}

For the sake of completeness we provide the reader with a sketch of the main ideas of our proof.

\begin{proof}
The main idea is that some elements $A$ of~$E_J$ are definable by the formula~$\varphi_N$, stating that  all products of its conjugates of the length~$N$ form a normal subgroup in~$E$ which does not coincide with~$E$.

If for a given~$A$ and some length~$N$ the formula $\varphi_N(A)$  is true, then 
 the minimal normal subgroup of~$E$, which contains~$A$, is a proper
subgroup of~$E$. 
As we know, every proper normal subgroup of~$E$ is
contained in~$E_J$, therefore $A\in E_J$. 
For rather a big~$N$ (depending only of~$\Phi$) the formula $\varphi_N$ 
 characterises elements of~$E_J$, maybe not all of them, but at least all generators $A=x_\alpha(u), u\in J$.

Let us fix the minimal natural~$N$ such that if for some~$A$ the
sentence $\varphi_N(A)$ does not hold, then for this~$A$ no sentence
$\varphi_p(A)$, $p> N$, holds.

Taking into account this~$N$ let us find  natural~$M$ such that:

(1)  products $X_1\dots X_k$, $k\leqslant M$ of elements of~$E$,
satisfying $\varphi_N(X)$,  form a subgroup in~$E$;

(2) this subgroup is normal;

(3) this subgroup is not trivial;

(4) this subgroup does not coincide with the whole group~$E$.

Clear that all these assertions are first order definable.

This number $M$ means  that every
element of the group $E_J$ is generated by not more than
$M$~elements $x_\alpha(u)$, $u\in J$.

Now the formula
$$
 Normal_{M,N}(X):= \bigwedge_{i=0}^M \exists X_1\dots
\exists X_i(\varphi_N(X_1)\land \dots \land \varphi_N(X_i)\land
X=X_1\dots X_i)
$$
defines in~$E$ the subgroup~$E_J$.
\end{proof}

\medskip

We proved that 
 the subgroup $E_J$ is absolutely definable in the
group~$E$, then, taking the quotient group  $E/E_J$, we obtain the Chevalley
group $\widetilde E\cong E_{\ad}(R/J)$, i.e., a Chevalley group
over field, which is also absolutely interpretable in $E=E_{\ad}(\Phi,R)$ (and also in the initial $G_\pi (\Phi,R)$).

\medskip

Recall that we have a root system~$\Phi$ of rank $>1$. The set of
simple roots is denoted by~$\Delta$, the set of positive roots is
denoted by~$\Phi^+$. The subgroup $U=U(R)$ of the Chevalley group
$G$ ($E$) is generated by elements $x_\alpha(t)$, $\alpha\in
\Phi^+$, $t\in R$, the subgroup $V=V(R)$ is generated by elements
$x_{-\alpha}(t)$, $\alpha\in \Phi^+$ $t\in R$.

For invertible $t\in R^*$ by $w_\alpha(t)$ we denote
$x_\alpha(t)x_{-\alpha}(-t^{-1}) x_\alpha (t)$, by $h_\alpha(t)$ we
denote $w_\alpha(t)w_\alpha(1)^{-1}$.

The group $H=H(R)$ is generated by all $h_\alpha(t)$, $\alpha\in
\Phi$, $t\in R^*$.

The strict proof of the following proposition can be found in~\cite{Bunina-local} as well. The first  assertion is well-known, but we also need assertions (2) and (3) for our purpose.

\begin{proposition}[Gauss decomposition and first order formulas]\label{gauss}

\emph{(1)} Every element $x$ of a Chevalley group $G$ $($respectively, of its elementary subgroup $E)$ over a
local ring~$R$ can be represented in the form
$$
x=utvu' \quad (\text{respectively, }x=uhvu'),
$$
where $u,u'\in U(R)$, $v\in V(R)$, $t\in T(R)$, $h\in H(R)$;

\emph{(2)} There exists a first order formula $\varphi(\dots)$ of the ring language and $6n+2l$ arguments, such that for decompositions $x_1=u_1t_1v_1u_1'$ and
$x_2=u_2t_2v_2u_2'$, where
\begin{align*}
u_i&= x_{\alpha_1}(t_1^{(i)})\dots x_{\alpha_n}(t_n^{(i)}),\\
u_i'&= x_{\alpha_1}(s_1^{(i)})\dots x_{\alpha_n}(s_n^{(i)}),\\
v_i&= x_{-\alpha_1}(r_1^{(i)})\dots x_{-\alpha_n}(r_n^{(i)}),\\
t_i&= h_{\alpha_1}(\xi_1^{(i)})\dots h_{\alpha_l}(\xi_l^{(i)}),\quad
i=1,2,
\end{align*}
this formula
\begin{multline*}
\varphi(t_1^{(1)},\dots,t_n^{(1)},t_1^{(2)},\dots,t_n^{(2)},
s_1^{(1)},\dots,s_n^{(1)},s_1^{(2)},\dots,s_n^{(2)},\\
r_1^{(1)},\dots,r_n^{(1)},r_1^{(2)},\dots,r_n^{(2)},
\xi_1^{(1)},\dots,\xi_n^{(1)},\xi_1^{(2)},\dots,\xi_n^{(2)}),
\end{multline*}
is true if and only if
$$
x_1=x_2;
$$

\emph{(3)} Similarly, there exists a first order formula $\psi(\dots)$ of the ring language and $9n+3l$ arguments, such that for decompositions $x_1=u_1t_1v_1u_1'$,
$x_2=u_2t_2v_2u_2'$ and $x_3=u_3t_3v_3u_3'$, where
\begin{align*}
u_i&= x_{\alpha_1}(t_1^{(i)})\dots x_{\alpha_n}(t_n^{(i)}),\\
u_i'&= x_{\alpha_1}(s_1^{(i)})\dots x_{\alpha_n}(s_n^{(i)}),\\
v_i&= x_{-\alpha_1}(r_1^{(i)})\dots x_{-\alpha_n}(r_n^{(i)}),\\
t_i&= h_{\alpha_1}(\xi_1^{(i)})\dots h_{\alpha_l}(\xi_l^{(i)}),\quad
i=1,2,3,
\end{align*}
this formula $
\psi(t_1^{(i)},\dots,t_n^{(i)}, s_1^{(i)},\dots,s_n^{(i)},
r_1^{(i)},\dots,r_n^{(i)}, \xi_1^{(i)},\dots,\xi_n^{(i)}),$
holds if and only if
$$
x_3=x_1\cdot x_2.
$$
\end{proposition}

\section{Regular bi-interpretability  of Chevalley groups over local rings}\leavevmode

Now we will prove Theorem~\ref{th-reg-bi-interpret}.
Our  proof  consists of three steps:

(1) Prove that all elementary unipotent subgroups $X_\alpha =\{ x_\alpha(t)\mid t\in R\}$ are definable in $G= E_{\ad}(\Phi, R)$ with parameters $\overline x=\{ x_\alpha(1)\mid \alpha\in \Phi\}$.

(2) Prove that a Chevalley groups $G_\pi(\Phi, R)$ (or an elementary Chevalley group $E_\pi(\Phi,R)$) is bi-interpretable with the ring $R$ with parameters $\overline x=\{ x_\alpha(1)\mid \alpha\in \Phi\}$.

(3) Prove that our parameters $\overline x$ are definable, i.\,e., bi-interpretability is regular.

We will allocate one subsection for each step.

\subsection{Definability of elementary unipotent subgroups}\leavevmode

Definability of elementary unipotent subgroups $X_\alpha$ was proved in~\cite{Segal-Tent} for  integral domains, which include fields.
So we can suppose that all groups $X_\alpha$ are definable in the group $E_{\ad} (\Phi,R)$, if $R$ is an arbitrary field.

Let us show that a root subgroup $X_\alpha$ over a local ring is also definable.

We will define it as the intersection of the inverse image of a root subgroup of the Chevalley group $E'=E_{\ad}(\Phi, R/\Rad R)$ over the residue field  under canonical homomorphism and the set 
$$
\{ AB \mid \forall x \, ([x,x_\alpha(1)]=1 \Longleftrightarrow [A, x]=1\land [B,x]=1)\}.
$$

To begin with, let us understand which elements lie in the inverse image of the root subgroup.

Let  $x_{\alpha}(\overline t) \in E'$, we denote some of its inverse image under canonical homomorphism by~$g$. In $E=E_{\ad}(\Phi,R)$ there exists a Gauss decomposition (see Proposition~\ref{gauss} above). Using it we can represent   $g$ as the product of a bounded number of elements $x_{\alpha_{i}}(t)$ of the special form. Ordering  positive roots, without loss of  generality, we can assume that $\alpha$ is positive and it is $\alpha_{1}$. We have:
$$
 g = x_{\alpha_{1}}(r_{1})\ldots x_{\alpha_{n}}(r_{n}) h x_{-\alpha_{1}}(s_{1})\ldots x_{-\alpha_{n}}(s_{n})x_{\alpha_{1}}(t_{1})\ldots x_{\alpha_{n}}(t_{n}),
$$
where  $\alpha_{i}$ are positive roots, $r_{i}, s_{i},t_i \in R$, $h\in T(R)$. The image of~$g$ under canonical homomorphism has the form  (by the homomorphism properties):
$$
x_{\alpha_{1}}(\overline{r_{1}})\ldots x_{\alpha_{n}}(\overline{r_{n}}) \overline{h} x_{-\alpha_{1}}(\overline{s_{1}})\ldots x_{-\alpha_{n}}(\overline{s_{n}})x_{\alpha_{1}}(\overline{t_{1}})\ldots x_{\alpha_{n}}(\overline{t_{n}})=x_{\alpha_{1}}(\overline t),
$$
where $\overline{r_{i}}, \overline{s_{i}}, \overline t_i \in R/J$. Let us move elements with positive roots to the right side of equality:
$$
\overline{h}x_{-\alpha_{1}}(\overline{s_{1}})\ldots x_{-\alpha_{n}}(\overline{s_{n}}) =  x_{\alpha_{n}}(\overline{r_{n}})^{-1}\ldots x_{\alpha_{1}}(\overline{r_{1}})^{-1}x_{\alpha_{1}}(\overline t)x_{\alpha_{n}}(\overline{t_{n}})^{-1}\ldots x_{\alpha_{1}}(\overline{t_{1}})^{-1}.
$$
We obtained the equality  $tv = u$, where $t\in T$, $v \in V$, $u \in U$, where $V$ and $U$ are the subgroups, generated by negative and positive root unipotents, respectively.
Since  $TV \;\cap\; U = 1$, then both sides are equal to~$1$, that is, the Gauss decomposition of the original element consists only of the torus element, root unipotents with positive roots and root unipotents with negative roots, which arguments belong to the radical, so we can assume that
 $$
g=x_{\alpha_{1}}(r_{1})\ldots x_{\alpha_{n}}(r_{n})h \cdot x_{-\alpha_{1}}(s_{1})\ldots x_{-\alpha_{n}}(s_{n}) ,\quad s_1,\dots, s_n\in J,
$$
and its image is
 $$
\overline h \cdot x_{\alpha_{1}}(\overline{r_{1}})\ldots x_{\alpha_{m}}(\overline{r_{m}})=x_{\alpha_1}(\overline t),
$$
therefore
$$
g=x_{\alpha_{1}}(r_{1})\ldots x_{\alpha_{n}}(r_{n})h\cdot x_{-\alpha_{1}}(s_{1})\ldots x_{-\alpha_{n}}(s_{n}) ,\quad r_2,\dots, r_n, s_1,\dots, s_n\in J, \overline h=1.
$$
 
Thus we proved definability of $X_{\alpha_{1}}E_{J}$. In order to prove finally that the subgroup $X_{\alpha_{1}}(R)$ remains definable, we will prove the following important proposition:

\begin{proposition}
 Let $G_{\alpha_{1}} = \{g \in E \;|\; C_{G}(g) = C_{G}(x_{\alpha_{1}}(1)) \}$. Then
$$
X_{\alpha_{1}}(R) \supseteq  X_{\alpha_{1}}(R)E_{J} \;\cap\; G_{\alpha_{1}} \supseteq X_{\alpha_1}(R^*).
$$
\end{proposition}

\begin{proof}
Clear that if $g\in X_{\alpha_1} (R^*)$, then $g\in G_{\alpha_1}$, therefore we will prove only the first inclusion.

Let $g \in X_{\alpha_{1}}(R)E_{J} \cap G_{\alpha_{1}} $.
Let us recall once again the type of elements $X_{\alpha_{1}}(R)E_{J}$ we have already proved :
$$
g=x_{\alpha_{1}}(r_{1})\ldots x_{\alpha_{n}}(r_{n})tx_{-\alpha_{1}}(s_{1})\ldots x_{-\alpha_{n}}(s_{n}) ,\quad r_2,\dots, r_n, s_1,\dots, s_n\in J.
$$
Note that, unlike the original Gauss decomposition, such a decomposition is uniquely defined. Indeed, suppose that
$$ 
u_{1}t_{1}v_{1} = u_{2}t_{2}v_{2}.
$$
If we move all the positive roots to one side:
$t_{1}v_{1}v_{2}^{-1}t_{2}^{-1} = u_{1}^{-1}u_{2}$, then since
$TV \;\cap\; U = 1$, so this decomposition is unique.

Let us stock up on an important formula
$$
x_\gamma(1) x_{-\gamma}(s) x_{\gamma} (1)^{-1}=h_\gamma \left(\frac{1}{1-s}\right) x_\gamma (s^2-s) x_{-\gamma}\left(\frac{s}{1-s}\right)\text{ for all }\gamma\in \Phi
$$
(it is checked directly trough a representation by matrices from $\SL_2$).
\medskip

We always suppose that the roots $\alpha_1,\dots, \alpha_n$ are ordered by their majority (first there are simple roots, then their sums, at the end there is  the major root).

Consider a root $\beta \in \Phi$ such that  $\alpha_{1} + \beta \notin \Phi$, and the element $x_{\beta}(1)$. Since $g \in G_{\alpha_{1}} $, then it commutes with $x_{\beta}(1)$. It means that  $g^{x_{\beta}(1)} = g$. Let us consider, how conjugation by this element acts on~$g$ and its separate multipliers (let for certainty $\beta$ be positive):
\begin{multline*}
g=g^{x_{\beta}(1)} = x_{\alpha_{1}}(r_{1})^{x_{\beta}(1)}\ldots x_{\alpha_{n}}(r_{n})^{x_{\beta}(1)}t^{x_{\beta}(1)}x_{-\alpha_{1}}(s_{1})^{x_{\beta}(1)}\ldots x_{-\alpha_{n}}(s_{n})^{x_{\beta}(1)}=\\
= x_{\alpha_{1}}(r_{1}) \dots ( x_{\alpha_i}(r_i)x_{\alpha_i+\beta}(c_i r_i)\dots ) \dots x_{\alpha_n}(r_n) \cdot (x_\beta(c)  t)\cdot\\
\cdot  (x_{-\alpha_1}(s_1)x_{-\alpha_1+\beta}(d_1 s_1)\dots) \dots \left(h_\beta\left( \frac{1}{1-s_{-\beta}}\right)x_\beta(s_{-\beta}^2-s_{-\beta})x_{-\beta}\left( \frac{s_{-\beta}}{1-s_{-\beta}}\right)\right)\dots\\
\dots (x_{-\alpha_n}(s_n)x_{-\alpha_n+\beta}(d_ns_n)\dots).
\end{multline*}
Let us  analyze the  obtained equality.

Note that to the left of $t$ are only unipotent elements with positive roots, that is, an element of~$U$. To the right of~$t$ we see the element of the torus $h_\beta\left( \frac{1}{1-s_{-\beta}}\right)$, which we can rearrange directly to~$t$ by changing only the arguments of the elements $x_\gamma(\cdot)$, we see unipotent elements with negative roots, as well as such $x_\gamma(\cdot)$ with $\gamma\in\Phi^+$, which, when moving to the left towards $T$ and $U$, cannot meet unipotents with the opposite root. This means that it is possible to rearrange all unipotents with positive roots that are to the right of~$t$, to the left of~$t$ so that $t$ and $h_\beta\left( \frac{1}{1-s_{-\beta}}\right)$ will not change. After that, the conjugate element will be written in the form of $UTV$, that is
$$
t=t\cdot h_\beta\left( \frac{1}{1-s_{-\beta}}\right), 
$$
so $s_{-\beta}=0$.

Thus, $x_\beta(\cdot)$ cannot appear in the right of~$t$ part of the expression during conjugation, while $x_\beta(\cdot)$ cannot appear in the left part of $x_\beta(\cdot)$ is only in the right place: $x_\beta (r_\beta)$. So, in the expression $x_\beta(c)$, which is obtained by conjugating the element $t$ with the element $x_\beta(1)$, there must be $c=0$, that is
$$
[t,x_\beta(1)]=1.
$$

Let us consider the part~$V$ after conjugation.

Conjugating $x_{-\alpha_i}(s_i)$, where $\alpha_i$ is older than~$\beta$, we can obtain additional unipotents from~$V$; but if $\alpha_i$ is younger then~$\beta$, then we can obtain additional unipotents from~$U$.

Consider the oldest root $\alpha_i$, for which $-\alpha_i+\beta\in \Phi$. Conjugating the corresponding unipotent we have
$$
x_{-\alpha_i}(s_i)^{x_\beta(1)}= x_{-\alpha_i}(s_i) x_{-\alpha_i+\beta} (c \cdot s_i) x_{-\alpha_i+2\beta} (\dots)\dots,
$$
where $c$ is an invertible integer (since it could be equal only to~$\pm 1, \pm 2$ or $\pm 3$, but in the second case we deal with root systems ${\mathbf B}_l,{\mathbf C}_l,{\mathbf F}_4, {\mathbf G}_2$ and $1/2\in R$ by our assumption, in the third case we deal with the root system ${\mathbf G}_2$ and also $1/3\in R$ by our assumption).
By the choice of $\alpha_i$ an element $x_{-\alpha_i+\beta} (c \cdot s_i)$ can no longer appear from any of the conjugates, therefore $c s_i=s_i=0$. By the next step let us consider the next (by majority) root  $\alpha_j$ such that $-\alpha_j+\beta\in \Phi$, and we will come to the same conclusion $s_j=0$. 

Thus, we will consistently understand that for any root~$\alpha_i$ for which $-\alpha_i+\beta\in\Phi$ the condition $s_i=0$ takes place. So, the element $x_\beta(1)$ commutes with all non-unit unipotents located to the right of~$t$, that is
$$
g=g^{x_{\beta}(1)} = x_{\alpha_{1}}(r_{1})^{x_{\beta}(1)}\ldots x_{\alpha_{n}}(r_{n})^{x_{\beta}(1)}tx_{-\alpha_{1}}(s_{1})\ldots x_{-\alpha_{n}}(s_{n}),
$$
therefore
$$
x_{\alpha_{1}}(r_{1})^{x_{\beta}(1)}\ldots x_{\alpha_{n}}(r_{n})^{x_{\beta}(1)}=x_{\alpha_{1}}(r_{1})\ldots x_{\alpha_{n}}(r_{n}),
$$
and we will come to the conclusion by the same arguments that $r_i=0$ for all $i$ such that $\alpha_i+\beta\in \Phi$.

Therefore for all roots $\beta\in \Phi$ such that $\beta+\alpha_1\notin \Phi\cup \{ 0\}$ if $\gamma+\beta\in \Phi\cup \{ 0\}$, then the corresponding unipotent $x_\gamma(t)=1$, i.\,e., $t=0$. 

We will show that in most root systems in this way we will be able to ``delete'' all root unipotents except the first one. We divide all root systems into three categories: simply laced root systems, systems with double relations, and ${\mathbf G}_2$.

Let us start with the ${\mathbf G}_2$ root system, so that we don't have to go back to it.

\begin{lemma}\label{G2}
Let $\Phi={\mathbf G}_2$,  $B = \{ \beta \in \Phi \;| \; \alpha_{1} + \beta \notin \Phi \cup \{ 0\} \}$. By deleting from the list of roots of the system $\Phi$ all elements of $\gamma$ such that $\exists \beta \in B\, (\beta+\gamma\in \Phi\cup \{ 0\})$,
we delete all roots except~$\alpha_1$.
\end{lemma}

\begin{proof}
Since $\alpha_1$ can be any of the roots of the system, we will first prove the lemma for $\alpha_1$ a short, and then for a long one.

If $\alpha_1$ is a short simple root,  then the system consists of the roots
$$
\pm \alpha_1, \pm \alpha_2, \pm (\alpha_1+\alpha_2), \pm (2\alpha_1+\alpha_2), \pm (3\alpha_1+\alpha_2), \pm (3 \alpha_1+2\alpha_2).
$$
In this case 
$$
B=\{ \alpha_1, -\alpha_2, 3\alpha_1+\alpha_2, \pm (3\alpha_1+2\alpha_2)\}.
$$
The roots $-\alpha_1, \alpha_2, -3\alpha_1-\alpha_2, \pm (3\alpha_1+2\alpha_2)$ are deleted, since they are opposite to the roots from~$B$; the roots $\pm (\alpha_1+\alpha_2), \pm (2\alpha_1+\alpha_2)$ are deleted according to the root~$\alpha_1\in B$, the root $-\alpha_2$ is deleted according to the root $3\alpha_1+2\alpha_2\in B$, $3\alpha_1+\alpha_2$ is deleted according to $-3\alpha_1-2\alpha_2\in B$. 
Therefore all roots except  $\alpha_1$ are deleted. 

In the case when  $\alpha_1$ is a long simple root, the consideration is absolutely similar (and even easier).
\end{proof}

\begin{lemma}\label{ADE}
Let $\Phi$ be a simply laced root system, $B$ be the same as in the previous lemma.
Again from the set of all root of  $\Phi$ all $\gamma$ such that $\exists \beta \in B\, (\beta+\gamma\in \Phi\cup \{ 0\})$ are all roots except~$\alpha_1$.
\end{lemma}

\begin{proof}
  To begin with, we note that all the roots $-\alpha_{2},\ldots,-\alpha_{l}$ belong to $B$, since for simple root $\alpha_{i}, \alpha_{j}$ the difference $\alpha_{i} - \alpha_{j} \notin \Phi$. 

For any non-simple positive root  $\theta$ there exists a simple root  $\gamma$ such that $\theta-\gamma \in \Phi^{+}$, therefore  $\theta$ can be deleted, since $-\gamma \in B$.

 The exception is the case when $\gamma=\alpha_{1}$. This case means that there exists a pair of positive roots $\{ \alpha_1, \beta=\theta-\alpha_1\}$, for which their sum is a root. It is clear that in a simply laced root system these roots generates the subsystem~${\mathbf A}_2$.
Then $\beta+\alpha_1\in \Phi$, $\beta+2\alpha_1\notin \Phi$, therefore $\theta=\beta+\alpha_1,-\beta\in B$, and consequently since $\alpha_1=\theta + (-\beta)\in \Phi$, then both roots are deleted.
Thus $\theta$ (any non-simple positive root) is deleted. 

All simple positive roots (except $\alpha_1$) are deleted, since they are opposite to the roots from~$B$. Therefore all positive roots except~$\alpha_1$ are deleted.

It is clear that the root $-\alpha_1$ is deleted.
Concerning other negative roots we mention that 
 any root  which is not collinear to~$\alpha_1$ can be positive under another choice of ordering (but with $\alpha_1$ simple). Therefore all roots except~$\alpha_1$ can be deleted.
\end{proof}

\begin{lemma}\label{h}
For any root system~$\Phi$ if $\langle \beta, \alpha_1\rangle =0$, $g \in X_{\alpha_{1}}(R)E_{J} \cap G_{\alpha_{1}}$, 
$$
g=x_{\alpha_1}(s_1)\dots x_{\alpha_n}(s_n) t x_{-\alpha_1}(r_1)\dots x_{-\alpha_n}(r_n),
$$
then
$$
[h_\beta( a) , x_{\alpha_i} (s_i)]=[h_\beta (a),x_{-\alpha_i}(r_i)]=1\text{ for all }i=1,\dots, n\text{ and }a\in R^*.
$$
\end{lemma}

\begin{proof}
Since
$$
h_\alpha (t) x_\gamma (s) h_\alpha (t)^{-1}=x_\gamma (t^{\langle \alpha,\gamma\rangle} \cdot s),
$$
then if $\langle \beta,\alpha_1\rangle =0$, then $h_\beta(a)$ commutes with $x_{\alpha_1}(1)$, consequently $h_\beta(a)$ commutes with~$g$. Then
$$
g=g^{h_\beta(a)}=x_{\alpha_1}(a^{\langle \beta,\alpha_1\rangle }s_1)\dots x_{\alpha_n}(a^{\langle \beta,\alpha_n\rangle }s_n) t x_{-\alpha_1}(a^{-\langle \beta,\alpha_1\rangle }r_1)\dots x_{-\alpha_n}(a^{-\langle \beta,\alpha_n\rangle }r_n),
$$
From the uniqueness of the decomposition for~$g$ we have
$$
[h_\beta( a) , x_{\alpha_i} (s_i)]=[h_\beta (a),x_{-\alpha_i}(r_i)]=1\text{ for all }i=1,\dots, n.
$$
\end{proof}

It remains for us to consider root systems with double relations (recall that in this case we suppose $1/2\in R$).

As above,  we can learn to delete only positive roots, since only the root $-\alpha_1$ cannot be made positive under any ordering in which~$\alpha_1$ is positive, however, $-\alpha_1$ is deleted due to the fact that $\alpha_1\in B$.

Consider an arbitrary positive root~$\beta$. If $\alpha_1+\beta\in \Phi$, then $\beta$ is deleted by~$\alpha_1$, therefore let us suppose that $\alpha_1+\beta\notin \Phi$. If $-\beta +\alpha_1\notin \Phi$, then $-\beta\in B$ and $\beta$ is deleted by~$-\beta$, so we can suppose that $\beta-\alpha_1\in \Phi$. Thus we can suppose that there exist roots  $\alpha_1$ and $\gamma=\beta-\alpha_1$ such that $\alpha_1+\gamma$ is a root, and $2\alpha+\gamma$ is not a root. It is clear that if the  roots $\alpha_1$ and $\gamma$ are enclosed in the system~${\mathbf A}_2$, then everything is good (see Lemma~\ref{ADE}), but if they are enclosed only to the system~${\mathbf B}_2$, then we need to have an additional consideration.

To finish all the cases, as we see, it remains simply to consider separately the root system~$B_2$ and prove that it is possible to delete all the roots in it (perhaps not only using the set~$B$, but also using the lemma~\ref{h}).

In the root system ${\mathbf B}_2$ there are roots
$$
\pm \mu, \pm \nu, \pm (\mu+\nu), \pm (2\mu +\nu),
$$
where the simple root $\mu$ is short, and $\nu$ is long.

If $\mu=\alpha_1$, then $\mu, -\nu, 2\mu+\nu\in B$, therefore $-\mu, \nu, \pm (\mu+\nu)$ and $-(2\mu+\nu)$ are deleted. The roots $-\nu$ and $2\mu+\nu$  can not be deleted according to~$B$, therefore we will use  $h_\alpha(-1)$.

Since the roots $\mu$ and $\mu+\nu$ are orthogonal, then by Lemma~\ref{h} we have
$$
[h_{\mu+\nu}(-1), x_{-\nu}(s)]=[h_{\mu+\nu}(-1),x_{2\mu+\nu}(r)]=1\text{ for  } s,r \text{ from decomposition of  }g.
$$
Since $\langle \mu+\nu, -\nu\rangle =-1$, then
$- s =s$, therefore $s=0$, what was required. Similarly for the root $2\mu +\nu$.

If $\nu=\alpha_1$, then $\nu, -\mu, \mu+\nu,\pm(2\mu+\nu)\in B$, therefore $\pm \mu, -\nu, \pm (\mu+\nu)$ and $\pm(2\mu+\nu)$ are deleted, what was required.

\end{proof}

Therefore commuting with all suitable  $x_{\beta}(1)$ and $h_\gamma(a)$,  our element comes to the form  $tx_{\alpha_{1}}(r_{1})$,  where  $t$ commutes with all $x_\beta(1)$ for  $\beta \in B$, which clearly gives that $t$ belongs to the group's center. Since for an adjoint elementary Chevalley group its center is trivial,  $t=1$ and the proposition is proved.

Therefore we proved the following theorem:

\begin{theorem}\label{th_define_rings}
If  $E=E_{\ad}(\Phi, R)$ is an elementary adoint Chevalley group of rank $> 1$, $R$ is a local ring $($with $\frac{1}{2}$ for the root systems ${\mathbf B}_l, {\mathbf C}_l, {\mathbf F}_4, {\mathbf G}_2$ and with  $\frac{1}{3}$ for  ${\mathbf G}_{2})$, then any root subgroup  $X_{\alpha}$, $\alpha\in \Phi$, is definable in~$E$ with parameters.
\end{theorem}

\subsection{Bi-interpretability with parameters of $G(R)$ and $R$}\leavevmode

Now we will concentrate to prove that a group $G_\pi (\Phi,R)$ (respectively, $E_\pi(\Phi,R)$) and the corresponding local ring $R$ are bi-interpretable (with parameters~$\overline x$), or more strictly that the pair $\langle G_\pi (\Phi,R), \overline x\rangle$ $(\langle E_\pi (\Phi,R), \overline x\rangle)$ and the ring~$R$ are absolutely bi-interpretable.

Since in the paper~\cite{Segal-Tent} of Segal and Tent the similar theorem is proved for integral domains, we will use it and notice that the similar arguments work in our case.

Following \cite{Segal-Tent} a bi-interpretation between $R$ and $G(R)$ has several steps, some of them are already proved, some of them are similar to Segal--Tent, some of them are clear.

\smallskip

{\bf Step 1.} An interpretation of $E(R)$ in $G(R)$ is a very simple absolute interpretation of the  bounded commutant group in a group. Also if we speak about absolute interpretation of $\langle E_\pi (\Phi,R), \overline x\rangle$ in  $\langle G_\pi (\Phi,R), \overline x\rangle)$, the parameters in one group coincide with the parameters in another one (since they always belong to the commutant).

\smallskip

{\bf Step 2.} An interpretation of $E_{\ad} (R)$ in $E(R)$ is also a very simple absolute interpretation of the quotient group of a group by it center. If we consider an interpretation of $\langle E_{\ad} (\Phi,R), \overline {\widetilde x}\rangle$ in $\langle E_\pi (\Phi,R), \overline x\rangle$, the parameters $ \overline {\widetilde x}$ are the quotient classes of the initial $\overline x$.

\smallskip

{\bf Step 3.}  An interpretation of $R$ in $E_{\ad} (R)$ (as in~\cite{Segal-Tent})  consists of an identification
of~$R$ with a definable abelian subgroup $R'$ of  $E_{\ad} (R)$ such that addition $\oplus$ in $R'$ is
the group operation in $E_{\ad} (R)$; and multiplication $\otimes$ in $R'$ is definable in $E_{\ad} (R)$. We will not repeat it here, since is absolutely the same arguments as in \cite{Segal-Tent} and the papers about Matsev-type theorems.

\smallskip

{\bf Step 4.} An interpretation of $G'\cong G (R)$ (or $G'\cong E(R)$)  in~$R$: according to the Gauss decomposition every element $g\in G'$  can be represented as the product of four elements $u_1 \in U$, $t\in T$, $v\in V$, $u_2\in U$,  
where $u_1,v,u_2$ uniquely correspond to vectors of the length~$n$ of elements from~$R$, $t$ uniquely corresponds to a vector of the length~$l$ of elements from~$R^*$.  Since equality and multiplication of two elements expressed by  these four vectors is first order defined, we have absolute interpretation $G'$ of $G (R)$ (or $E(R)$)  in~$R$.

For the group $E(R)$ we always use this type of interpretation, but for the group G(R) it is possible also to repeat all arguments from~\cite{Segal-Tent}, since they do not depend of any properties of a given ring.

\smallskip

{\bf Step 5.} As in \cite{Segal-Tent} we define an isomorphism $G(R)\to G(R')$ or $E(R)\to E(R')$, it is definable according to Proposition~\ref{gauss}.

\medskip

{\bf Step 6.} An isomorphism $R\to R'$ is  the same as in the paper~\cite{Segal-Tent}.

Therefore we proved the following important theorem:

\begin{theorem}\label{th_bi-interpret}
If  $G(R)=G_\pi (\Phi,R)$ $(E(R)=E_{\pi}(\Phi, R))$ is an (elementary)  Chevalley group of rank $> 1$, $R$ is a local ring  $($with $\frac{1}{2}$ for the root systems ${\mathbf B}_l, {\mathbf C}_l, {\mathbf F}_4, {\mathbf G}_2$ and with  $\frac{1}{3}$ for  ${\mathbf G}_{2})$, $\overline x=\{ x_\alpha(1)\mid \alpha\in \Phi\}$, then the pair $(G(R), \overline x)$ $($or $(E(R),\overline x))$ is absolutely bi-interpretable with~$R$.
\end{theorem}

\subsection{Definability of parameters}\leavevmode

In this section we are going to show that our parameters $\overline x=\{ x_\alpha (1)\mid \alpha \in \Phi\}$ are definable up to an automorphism of the Chevalley group $E_{\ad} (\Phi,R)$.

What do we know about these parameters?

1. $x_\alpha(1)$ and $x_\beta(1)$ are conjugate, if $|\alpha|=|\beta|$ and we know the order of some possible conjugating element.

2. We know all commutator relations between parameters.

3. We know that the centres of centralizers  of each $x_\alpha(1)$ with operations $\oplus$ and $\otimes$ introduced in the previous section are local rings (let us call them $X_\alpha (\overline x)$) isomorphic to each other.

4. Any element of the whole group is represented as the product $utvu'$, where $u,t,v,u'$ are defined trough~$X_\alpha$ as in Proposition~\ref{gauss}, the equality of two such representations is defined as in Proposition~\ref{gauss}, multiplication of two such elements is defined as in Proposition~\ref{gauss}.

All together properties 1--4 can be expressed by a first order formula $\Phi(\overline x)$.
In fact it means that there exists a ring $R'=(X_\alpha,\oplus,\otimes)$ such that $E_{\ad} (\Phi,R)\cong E_{\ad} (\Phi,R')$. By \cite{Bunina_Chev_local_aut} an isomorphism is the composition of a ring isomorphism (where $x_\alpha(t)\mapsto x_\alpha(t')$ for all $\alpha\in \Phi$ and $t\in R$) and some automorphism of the initial group $E_{\ad} (\Phi,R)$. This means that $R'\cong R$ and any parameters $\overline t$ satisfied the formula~$\Phi$ up to an automorphism of the group $E_{\ad} (\Phi,R)$ have the form $\overline x=\{ x_\alpha (1)\mid \alpha \in \Phi\}$, what was required.

Therefore bi-interpretability of the  Chevalley group $E_{\ad} (\Phi,R)$ is regular.

Finally we proved our main theorem:

\medskip 

\emph{
If  $G(R)=G_\pi (\Phi,R)$ $(E(R)=E_{\pi}(\Phi, R))$ is an (elementary)  Chevalley group of rank $> 1$, $R$ is a local ring  (with $\frac{1}{2}$ for the root systems ${\mathbf B}_l, {\mathbf C}_l, {\mathbf F}_4, {\mathbf G}_2$ and with  $\frac{1}{3}$ for  ${\mathbf G}_{2})$, then the group $G(R)$ (or $(E(R)$) is regularly bi-interpretable with~$R$.
}

\subsection{Elementary definability of Chevalley groups}\leavevmode

In this section we will show that regular bi-interpretability implies elementary definability.

\begin{proposition}\label{prop-regular->def}
If for some class of rings $\mathcal R$, closed under elementary equivalence,  and the class $\mathcal G =\{ G(R)\mid R\in \mathcal R\}$ (where $G(R)$ is any type of derivative groups: linear groups, Chevalley groups, automorphism groups, etc.) all the groups $G(R)$ are regularly interpretable with the corresponding rings~$R$ with the same interpretations, then the class $\mathcal G$ is elementarily definable in class of all groups, i.\,e., for any group $H$ such that $H\equiv G(R)$ for some $R\in \mathcal R$ there exists a ring $R'\equiv R$ such that $H\cong G(R')$.
\end{proposition} 

\begin{proof}
Let $H$ be some group such that $H\equiv G(R)$, $R\in \mathcal R$. Let $\varphi (x_1,\dots, x_m)$ be the formula defining  parameters $y_1,\dots, y_m$, required for bi-interpretation of $R$ in $G(R)$ (recall that these parameters are defined up to some automorphism of $G(R)$). The formula $\varphi(x_1,\dots, x_m)$ defines in~$H$ some parameters $h_1,\dots, h_m$, let us fix them.

Let now $R \cong \Gamma(G(R), \varphi)$ be a code interpreting $R$ in $G(R)$. Let us apply the same code to the group~$H$, we will obtain some ring $R'$, elementarily equivalent to~$R$.

Since $R'\equiv R$, then $R'\in \mathcal R$, so that $G(R')\in \mathcal G$. It means that the code $\Gamma(\cdot, \varphi)$ also defines in $G(R')$ the ring~$R'$, and then a code $\Delta (\cdot, \psi)$, which interprets $G(R)$ in~$R$, also interprets $G(R')$ in~$R'$.  
 
But there is a formula $\theta (\overline y, x, \overline z)$ in the language of groups
 such that for every tuple~$\overline p$ satisfying $\varphi(\overline z)$ in~$G(R)$ (or in $G(R')$) the formula $\theta (\overline y,x, \overline p)$   defines in~$G(R)$ (in $G(R')$) the isomorphism $\overline \mu_{\Delta \circ \Gamma}:  (\Delta \circ \Gamma)(G(R), \overline p) \to G(R)$ (or $\overline \mu_{\Delta \circ \Gamma}':  (\Delta \circ \Gamma)(G(R'), \overline p') \to G(R')$). 

Since $H\equiv G(R)$, the same formula defines an isomorphism between $(\Delta\circ \Gamma)(H, \overline p'')$ and $H$, where $\overline p''$ is defined in~$H$ by the formula~$\varphi$. 

But the code $\Gamma(\cdot , \varphi)$ defines in $H$ a ring isomorphic to~$R'$, therefore this isomorphism is between a group, constructed by~$R'$ according to interpretation code~$\Delta$, and the group~$H$. We know that $\Delta(R')\cong G(R')$, consequently, $G(R')\cong H$, as required.

\end{proof}

\begin{corollary}
The class of Chevalley groups over local rings is elementarily definable, i.\,e., if $G(R)=G_\pi (\Phi,R)$ is a Chevalley group of rank $> 1$, $R$ over a  local ring $R$  (with $\frac{1}{2}$ for the root systems ${\mathbf A}_2, {\mathbf B}_l, {\mathbf C}_l, {\mathbf F}_4, {\mathbf G}_2$ and with  $\frac{1}{3}$ for  ${\mathbf G}_{2})$  and for an arbitrary group $H$ we have $H\equiv G(R)$, then $H\equiv G_\pi(\Phi,R')$ for some local ring~$R'$, which is elementarily equivalent to~$R$.
\end{corollary}

\begin{proof}
By the results of~\cite{Bunina-local}  Chevalley groups of different types  (over infinite rings) are not elementarily equivalent, i.\,e., they cannot have different root systems or not isomorphic weight lattices. Therefore we can suppose that we consider the class $\mathcal G=\{ G(R) =G_\pi (\Phi, R)\mid R\text{ is a local ring }\}$, where $\Phi$ and $\pi$ are fixed.

Since by Theorem~\ref{th-reg-bi-interpret} the class $\mathcal G$ satisfies the conditions of Proposition~\ref{prop-regular->def}, then by this proposition it is elementarily definable.
\end{proof}

\medskip

{\bf Acknowledgements.}
Our sincere thanks go to Eugene Plotkin for very useful discussions regarding various aspects of this work and permanent attention to it.

\end{document}